\documentclass[11pt,reqno]{amsart}
\usepackage[tt=false]{libertine}
\usepackage{amssymb}
\usepackage{mathtools}
\usepackage[varbb]{newpxmath}
\usepackage[foot]{amsaddr}

\let\savedbigtimes\bigtimes
\let\bigtimes\relax
\usepackage{mathabx}
\let\bigtimes\savedbigtimes

\usepackage[margin=1in, bottom=1in]{geometry}
\usepackage{graphicx}
\usepackage{caption}
\usepackage{subcaption}
\usepackage{enumitem}
\usepackage{comment}
\usepackage{bbm}

\usepackage[usenames,dvipsnames]{xcolor}
\usepackage[colorlinks=true,
  linkcolor=teal!60!black,
  citecolor=PineGreen,
  urlcolor=RedViolet]{hyperref}
\hypersetup{bookmarksopen=true}

\usepackage[yyyymmdd,hhmmss]{datetime}

\usepackage{tikz}
\usetikzlibrary{decorations.pathreplacing,calligraphy}
\usetikzlibrary{calc}

\newtheorem{thm}{Theorem}[section]
\newtheorem{lem}[thm]{Lemma}
\newtheorem{ppn}[thm]{Proposition}

\newtheorem{conj}[thm]{Conjecture}

\theoremstyle{definition}

\theoremstyle{remark}
\newtheorem{rmk}[thm]{Remark}

\numberwithin{equation}{section}
\def\beq#1\eeq{%
    \begin{equation}%
    #1%
    \end{equation}%
}
\def\baln#1\ealn{%
    \begin{align}%
    #1%
    \end{align}%
}
\def\balnn#1\ealnn{%
    \begin{align*}%
    #1%
    \end{align*}%
}

\def\bbC{{\mathbb{C}}}

\def\bbE{{\mathbb{E}}}

\def\bbP{{\mathbb{P}}}

\def\bbR{{\mathbb{R}}}

\def\bbT{{\mathbb{T}}}

\def\ker{{\mathsf{ker}}}

\title{Resolution of Vinzant's Conjecture on Phase Retrieval Injectivity}

\author{Zhangsong Li}

\address[Z.~Li]{School of Mathematical Sciences, Peking University}
\email[Z.~Li]{ramblerlzs@pku.edu.cn}

\date{\today}

\begin{document}

\begin{abstract}
    Let $A\in\bbC^{N\times M}$ with $N=4M-5$ has independent standard complex Gaussian entries. We prove that the probability that the phase retrieval map generated by $A$ is injective is strictly smaller than $1$, and is $O(M^{-1})$ as $M\to \infty$. This proves a conjecture of Cynthia Vinzant, which was later restated by Afonso S. Bandeira in \cite{bandeira2026randomstrasse101}. The main result of this paper was obtained using generative AI, in particular GPT6-Astra, GPT 5.6-sol, and the Rethlas system.
\end{abstract}

\maketitle

\setcounter{tocdepth}{1}
\tableofcontents

\section{Introduction and main results}
\label{s:intro}

Let $N,M$ be two integers. Given a complex matrix $A \in \mathbb C^{N\times M}$, the phase retrieval problem aims to recover a vector $x \in \mathbb C^{M}$ from $|Ax|$, where $|\cdot|$ is the entrywise modulus. In other words, one has access to the modulus of $N$ linear measurements, and the goal is to recover $x \in \mathbb C^M$. Let $\mathbb T$ be the unit circle in $\mathbb C$, it is clear that one can only hope to recover $x$ modulo $\mathbb T$. More precisely, let $\mathbb C^M (\operatorname{mod} \mathbb T)$ be $\mathbb C^M$ modulo $\mathbb T$ (meaning that $x=(x_1,\ldots,x_M) \in \mathbb C^M$ and $y=(y_1,\ldots,y_M) \in \mathbb C^M$ are equivalent if and only if there exists $|z|=1$ such that $x=zy$), we say $A$ is injective for phase retrieval if the map
\begin{align}{\label{eq:phase-retrieval-map}}
    \Phi_A: x \in \mathbb C^M (\operatorname{mod} \mathbb T) \to |Ax|
\end{align}
is injective.

In \cite{bandeira2014saving}, the authors conjectured that $N \geq 4M-4$ linear measurements were necessary for injectivity over $\mathbb C$, and that generic $4M-4$ measurements were injective. While the second part of this conjecture was proven in \cite{conca2015algebraic}, the first was shown to be false when Cynthia Vinzant found an injective set of $11$ linear measurements in $\mathbb C^4$ \cite{vinzant2015small}. As a remedy, a refined version of this conjecture was later posed by Cynthia Vinzant (and stated again in \cite{bandeira2026randomstrasse101}), which we state as follows.
\begin{conj}{\label{conjecture:main}}
    Let $M\ge2$, $N=4M-5$, and let $A \in \mathbb C^{N\times M}$ be a matrix whose entries are drawn i.i.d.\ from the standard complex Gaussian distribution (a complex standard Gaussian has the law of $a+ib$ where $a,b \sim \mathcal N(0,\frac{1}{2})$ are independent). Let $p_M$ be the probability that the mapping $\Phi_A$ defined in \eqref{eq:phase-retrieval-map} is injective. We then have:
    \begin{enumerate}
        \item[(1)] $p_M<1$ for all $M$.
        \item[(2)] $\lim_{M \to\infty} p_M=0$.
    \end{enumerate}
\end{conj}
The goal of this work is to prove Conjecture~\ref{conjecture:main}. Out first result verifies Part~(1) of the conjecture.
\begin{thm}{\label{thm:main-1}}
    Let $M \geq 2$ and $N=4M-5$. Then, there is a nonempty open set $U \subset \mathbb C^{N\times M}$ such that every $A \in U$ is not injective for phase retrieval. In particular, this implies Part~(1) of Conjecture~\ref{conjecture:main}.
\end{thm}
\begin{rmk}
    Note that the case $M=2,N=3$ has already been established in \cite{bandeira2014saving}. Thus, in the later proof we will focus on $M \geq 3$. In addition, we note that the case $M=2^k+1$ was proved in \cite{conca2015algebraic}. We would like to remark that it seems the construction in our work differs significantly from that in \cite{conca2015algebraic}.
\end{rmk}
Our second result verifies Part~(2) of the conjecture.
\begin{thm}\label{thm:main-2}
    There exists a constant $C>0$ such that $p_M\le C M^{-1}$ for every $M \ge 2$. In particular, $\lim_{M\to\infty}p_M=0$.
\end{thm}

\subsection*{Acknowledgment}

We would like to thank Leheng Chen, Zihao Liu, and Haocheng Ju for introducing the Rethlas system to him. We also would like to thank the Rethlas team, namely Haocheng Ju, Guoxiong Gao, Jiedong Jiang, Bin Wu, Zeming Sun, Leheng Chen, Yutong Wang, Yuefeng Wang, Zichen Wang, Wanyi He, Peihao Wu, Liang Xiao, Ruochuan Liu, Bryan Dai, and Bin Dong, for their contributions to the development of Rethlas and its customized version used for the problem studied in this paper. 

\subsection*{Declaration of AI usage}

The main result of this work was obtained using generative AI. In particular, the proof of Theorem~\ref{thm:main-1} is obtained by the Rethlas system, and the proof of Theorem~\ref{thm:main-2} is obtained by GPT 6-Astra and GPT 5.6-Sol. Both theorems are prompted by Zhangsong Li without any mathematical inputs. We refer to \cite{ju2026automated} for a detailed introduction to the Rethlas system. Due to the limitations of generative AI, it is possible that we have overlooked some relevant references.

\section{Proof of Theorem~\ref{thm:main-1}}{\label{sec:proof-main-thm-1}}

This section is devoted to the proof of Theorem~\ref{thm:main-1}. We will first invoke a result in \cite{bandeira2014saving} which provides a reformulation of the injectivity of the map $\Phi_A$ defined in \eqref{eq:phase-retrieval-map}.
\begin{lem}[Lemma~9 in \cite{bandeira2014saving}]{\label{lem:equivalent-formulation}}
    Let $A \in \mathbb C^{N \times M}$, and write its rows as $a_1^{*},\ldots,a_N^{*}$, where $a_j \in \mathbb C^{M}$ and $a_j^{*}$ is the conjugate transpose of $a_j$. Define the real-linear map $\mathcal L_A$ on Hermitian matrices $Q \in \mathbb C^{M \times M}$ such that
    \begin{align}{\label{eq:def-mathcal-L-A}}
        \mathcal L_A(Q) = (a_1^* Q a_1,\ldots,a_N^* Q a_N) \,.
    \end{align}
    Then the phase retrieval map $\Phi_A$ in \eqref{eq:phase-retrieval-map} is injective if and only if $\mathsf{ker}(\mathcal L_A)$ contains no nonzero Hermitian matrix of rank at most $2$.
\end{lem}
Provided with Lemma~\ref{lem:equivalent-formulation}, to prove Theorem~\ref{thm:main-1}, it suffices to prove the following result.
\begin{lem}{\label{lem:main-1}}
    Let $M \geq 2$ and $N=4M-5$. Then, there is a nonempty open set $U \subset \mathbb C^{N\times M}$ such that for every $A \in U$, there exists a non-zero Hermitian matrix $Q(A) \in \mathbb C^{M\times M}$ with $\mathsf{rank}(Q(A)) \leq 2$ and $Q(A) \in \ker(\mathcal L_A)$.
\end{lem}
\begin{proof}
We will first construct a pair $(A_0,Q_0)$ such that $A_0 \in \mathbb C^{N\times M}$, $Q_0 \in \mathbb C^{M\times M}$ is rank-$2$ and Hermitian, and $Q_0 \in \mathsf{ker}(\mathcal L_{A_0})$. Then, we will generalize our construction to an open neighborhood of $A_0$ via the implicit function theorem.

We first construct $(A_0,Q_0)$. Define
\begin{equation}{\label{eq:def-Q-0}}
    Q_0 = \mathsf{diag}(1,-1,0,\ldots,0) \in \mathbb C^{M \times M} \,.
\end{equation}
It is clear that $Q_0$ is rank-$2$ and Hermitian. In addition, let $\mathbf e_1,\ldots,\mathbf e_{M-2}$ denote the standard basis of $\mathbb C^{M-2}$, and let $\mathbf 0$ denote the zero vector in $\mathbb C^{M-2}$. Define $A_0\in\mathbb C^{N\times M}$ with rows $a_i^*$, where the transposes $a_i^{\mathsf T}$ of the column vectors $a_i$ are the following tuples:
\begin{equation}{\label{eq:def-A-0}}
    \begin{aligned}
        &\Big\{ (1,1,\mathbf 0), \quad (1,-1,\mathbf 0), \quad (1,i,\mathbf 0) \Big\} \mbox{ and } \\
        &\Big\{ (1,1,\mathbf e_\ell), \quad (1,1,i\mathbf e_\ell), \quad (1,-1,\mathbf e_\ell), \quad (1,-1,i\mathbf e_\ell): 1 \leq \ell \leq M-2 \Big\} \,.
    \end{aligned}
\end{equation}
It is clear that for each of these column vectors $a\in \mathbb C^M$, we have $a^* Q_0 a=|a_1|^2-|a_2|^2=0$. Thus, $Q_0 \in \mathsf{ker}(\mathcal L_{A_0})$.

We now extend our construction to a small open neighborhood of $A_0$. For $s \in \mathbb R, b \in \mathbb C$ and $z,t \in \mathbb C^{M-2}$, define
\begin{equation}{\label{eq:def-matrix-D-C}}
    D(s,b)=
    \begin{pmatrix}
        1+\frac{s}{2} & b \\
        \overline{b} & -1+\frac{s}{2}
    \end{pmatrix}, \quad
    C(z,t)=
    \begin{pmatrix}
        z_1 & t_1 \\
        \vdots & \vdots \\
        z_{M-2} & t_{M-2} 
    \end{pmatrix} \,.
\end{equation}
For $|s|,|b|,|z|,|t|$ sufficiently small, $D(s,b)$ is invertible, and thus
\begin{equation}{\label{eq:def-Q(s,b,z,t)}}
    Q(s,b,z,t) = 
    \begin{pmatrix}
        D(s,b) & C(z,t)^* \\
        C(z,t) & C(z,t) D(s,b)^{-1} C(z,t)^*
    \end{pmatrix} \,.
\end{equation}
is an Hermitian matrix of rank $2$, and also $Q(0,0,\mathbf 0,\mathbf 0)=Q_0$. For a column vector $a=(u,v,w^{\mathsf T})^{\mathsf T}\in\mathbb C^M$ (we let $u,v \in \mathbb C$ and $w \in \mathbb C^{M-2}$), consider
\begin{align*}
    F_a(s,b,z,t) = a^* Q(s,b,z,t) a \,.
\end{align*}
Now, for any $A \in \mathbb C^{N \times M}$, write its rows as $a_1^*,\ldots,a_N^*$ with $a_i\in\mathbb C^M$, and define
\begin{align*}
    \Psi: \mathbb C^{N\times M} \times \mathbb R \times \mathbb C \times \mathbb C^{M-2} \times \mathbb C^{M-2} \to \mathbb R^{4M-5}, \quad \Psi(A,s,b,z,t)= ( F_{a_i}(s,b,z,t) )_{1 \leq i \leq N} \,.
\end{align*}
It is clear that both spaces $\mathbb R \times \mathbb C \times \mathbb C^{M-2} \times \mathbb C^{M-2}$ and $\mathbb R^{4M-5}$ have real dimension $N=4M-5$. In addition, we have $\Psi(A_0,0,0,\mathbf 0,\mathbf 0)=0$. We now argue that 
\begin{align}{\label{eq:diff-form-invertible}}
    \mathsf{D}_{s,b,z,t}\Psi(A_0,0,0,\mathbf 0,\mathbf 0) \mbox{ is invertible} \,.
\end{align}
Note that the lower-right block $CD^{-1}C^*$ has no linear term. Thus, for $a=(u,v,w^{\mathsf T})^{\mathsf T}$ where $|u|=|v|=1$ and for $|s|,|b|,|z|,|t|$ sufficiently small, we have
\begin{align*}
    F_a(s,b,z,t) = s + 2\mathsf{Re}(\overline{u}bv) + 2\mathsf{Re}(\overline{u} z^* w + \overline{v} t^* w) + O\left( |s|^2+|b|^2+|z|^2+|t|^2 \right) \,.
\end{align*}
Plugging \eqref{eq:def-A-0} into the above formula, we have 
\begin{align*}
    \Psi(A_0,s,b,z,t) = 
    \begin{pmatrix}
        s+2\mathsf{Re}(b) \\
        s-2\mathsf{Re}(b) \\
        s-2\mathsf{Im}(b) \\
        s+2\mathsf{Re}(b)+2\mathsf{Re}(z_1+t_1) \\
        s+2\mathsf{Re}(b)+2\mathsf{Im}(z_1+t_1) \\
        s-2\mathsf{Re}(b)+2\mathsf{Re}(z_1-t_1) \\
        s-2\mathsf{Re}(b)+2\mathsf{Im}(z_1-t_1) \\
        \vdots \\
        s+2\mathsf{Re}(b)+2\mathsf{Re}(z_{M-2}+t_{M-2}) \\
        s+2\mathsf{Re}(b)+2\mathsf{Im}(z_{M-2}+t_{M-2}) \\
        s-2\mathsf{Re}(b)+2\mathsf{Re}(z_{M-2}-t_{M-2}) \\
        s-2\mathsf{Re}(b)+2\mathsf{Im}(z_{M-2}-t_{M-2}) \\
    \end{pmatrix}
    + O\left( |s|^2+|b|^2+|z|^2+|t|^2 \right) \,.
\end{align*}
Thus, it is clear that \eqref{eq:diff-form-invertible} holds. Therefore, by the real implicit function theorem, for all $A$ in some open neighborhood $U$ of $A_0$, there are small parameters $(s(A),b(A),z(A),t(A))$ such that
\begin{align*}
    \Psi(A,s(A),b(A),z(A),t(A))=(0,\ldots,0) \,.
\end{align*}
Thus, the corresponding $Q(A)=Q(s(A),b(A),z(A),t(A))$ is rank-$2$, Hermitian, and satisfies $Q(A) \in \mathsf{ker}(\mathcal L_A)$.
\end{proof}

\section{Proof of Theorem~\ref{thm:main-2}}
\label{sec:proof-main-thm-2}

Throughout this section, $N=4M-5$, the rows of $A$ are $a_1^*,\ldots,a_N^*$ with $a_i\in\bbC^M$, and
\[
 \mathcal E_M=\{A:\Phi_A\text{ is injective}\},\qquad
 p_M=\bbP_0(\mathcal E_M),
\]
where $\bbP_0$ is the law of a matrix with independent standard complex Gaussian entries. We prove the required estimate for all sufficiently large $M$; enlarging the final constant then covers every $M\ge2$. All constants in this section are absolute and independent of $M$.
Vector norms are Euclidean; matrix norms with subscripts $\mathsf{op}$ and $\mathsf{F}$ are the operator and Frobenius norms, respectively.

We use two deterministic facts. By Lemma~\ref{lem:equivalent-formulation}, a nonzero rank-two Hermitian matrix in $\ker(\mathcal L_A)$ implies noninjectivity. In the indefinite case, if $Q=\lambda uu^*-\mu vv^*$ with $\lambda,\mu>0$ and orthonormal $u,v$, then $\sqrt\lambda\,u$ and $\sqrt\mu\,v$ are inequivalent signals with identical measurements. Also, for every invertible complex matrix $R$, the maps $\Phi_A$ and $\Phi_{AR}$ are either both injective or both noninjective, because $x\mapsto Rx$ is a bijection on $\bbC^M\ (\operatorname{mod}\bbT)$.

For completeness, $\mathcal E_M$ is measurable. Indeed, the set of Hermitian matrices with rank at most two and Frobenius norm one is compact. The minimum of $\|\mathcal L_A(Q)\|$ over this set is a continuous function of $A$, and Lemma~\ref{lem:equivalent-formulation} identifies $\mathcal E_M$ with the set where this minimum is positive.

The proof compares a probability law that favors an approximate rank-two ambiguity with a Gaussian law having the same covariance in a random two-dimensional subspace. We first control the difference of their likelihoods in $L^2$, and then use a quantitative implicit-function argument to produce an exact ambiguity under the first law.

\subsection{The planted law and its Gaussian reference}

Let $\gamma_d$ be standard complex Gaussian measure on $\bbC^d$, with density $\pi^{-d}e^{-\|z\|^2}$. Fix
\[
 \eta=\frac1{100},\qquad \delta=M^{-2},\qquad \varepsilon=M^{-50}.
\]
For $z=(z_1,z_2)\ne0$, set
\[
 S(z)=|z_1|^2+|z_2|^2,\qquad
 T(z)=\frac{|z_1|^2-|z_2|^2}{S(z)}.
\]
Under $\gamma_2$, $S$ has density $s e^{-s}\mathbbm{1}_{s>0}$, $T$ is uniform on $[-1,1]$, and $S,T$, and the two phases are independent. This follows by changing variables from the two independent exponential variables $|z_1|^2,|z_2|^2$.
Define
\begin{equation}\label{eq:g}
 g(z)=\frac{1}{\varepsilon c_\delta}
 \left(\eta+\frac{1-\eta}{S(z)}\right)
 \mathbbm{1}_{\{S(z)\ge\delta,\ |T(z)|\le\varepsilon\}},
 \qquad c_\delta=e^{-\delta}(1+\eta\delta).
\end{equation}
Its value at $z=0$ is zero. Direct integration gives $\int g\,d\gamma_2=1$. Under the probability measure $g\,d\gamma_2$, the two phases remain independent and uniform, $T$ is uniform on $[-\varepsilon,\varepsilon]$, and $S$ is independent of these variables with density
\begin{equation}\label{eq:radial}
 \frac{[(1-\eta)+\eta s]e^{-s}}{c_\delta}\mathbbm{1}_{s\ge\delta}.
\end{equation}
Consequently, this measure is centered, circular, and has covariance $v_M I_2$, where
\begin{equation}\label{eq:variance}
 v_M=\frac{(1+\eta)(1+\delta)+\eta\delta^2}{2(1+\eta\delta)}
 \longrightarrow \frac{1+\eta}{2}.
\end{equation}
Moreover, $\int e^{S/4}g\,d\gamma_2\le C$ and $0\le g\le C M^{52}$. The reference density is
\begin{equation}\label{eq:r}
 r(z)=v_M^{-2}\exp\bigl[-(v_M^{-1}-1)S(z)\bigr].
\end{equation}
Thus $r\,d\gamma_2$ is complex Gaussian with covariance $v_M I_2$. Both $g\,d\gamma_2$ and $r\,d\gamma_2$ have mean zero, vanishing pseudo-covariance, covariance $v_M I_2$, and a uniformly bounded exponential moment of $S$. For large $M$, $r$ is uniformly bounded.

Let $U\in\bbC^{M\times2}$ be a Haar-distributed orthonormal two-frame. Relative to the original Gaussian law $\bbP_0$ of $A$, define two likelihoods
\begin{equation}\label{eq:likelihoods}
 L_g(A)=\bbE_U\prod_{i=1}^N g(U^*a_i),\qquad
 L_r(A)=\bbE_U\prod_{i=1}^N r(U^*a_i).
\end{equation}
They both have expectation one. Denote their probability laws by $\bbP_g$ and $\bbP_r$. Conditionally on $U$, the reference rows have covariance $I_M+(v_M-1)UU^*$. Equivalently, under $\bbP_r$,
\[
 A=G R_U,\qquad R_U=I_M+(\sqrt{v_M}-1)UU^*,
\]
where $G$ is standard Gaussian, independent of $U$, and $R_U$ is invertible. The deterministic invariance above gives
\begin{equation}\label{eq:invariance}
 \bbP_r(\mathcal E_M)=p_M.
\end{equation}
We will prove the following two propositions.
\begin{ppn}\label{prop:plant}
Under the planted law, $\bbP_g(\mathcal E_M)=O(M^{-2})$.
\end{ppn}
\begin{ppn}\label{prop:L2}
The likelihoods in \eqref{eq:likelihoods} satisfy $\bbE_0(L_g-L_r)^2=O(M^{-1})$.
\end{ppn}
We first prove Theorem~\ref{thm:main-2} assuming Propositions~\ref{prop:plant} and \ref{prop:L2}.
\begin{proof}[Proof of Theorem~\ref{thm:main-2}]
Recall that $\bbP_0(\mathcal E_M)=\bbP_r(\mathcal E_M)=p_M$. Combining \eqref{eq:invariance}, Propositions~\ref{prop:L2} and \ref{prop:plant}, and Cauchy--Schwarz gives
\begin{align*}
 p_M
 &=\bbP_r(\mathcal E_M)\\
 &\le\bbP_g(\mathcal E_M)+\left|\bbE_0\bigl[(L_r-L_g)\mathbbm{1}_{\mathcal E_M}\bigr]\right|\\
 &\le C M^{-2}+\bigl(\bbE_0(L_r-L_g)^2\bigr)^{1/2}\sqrt{p_M}\\
 &\le C M^{-2}+\sqrt{\frac{C p_M}{M}}
 \le C M^{-2}+\frac{p_M}{2}+\frac{C}{2M}.
\end{align*}
Rearranging proves $p_M\le C'M^{-1}$ for all sufficiently large $M$. Since $p_M\le1$, increasing $C'$ covers the finitely many remaining integers $M\ge2$.
\end{proof}
The rest of the section is devoted to proving Propositions \ref{prop:plant} and \ref{prop:L2}.

\subsection{A Gaussian correlation estimate for two cones}

For a complex $2\times2$ matrix $K$ with $\|K\|_{\mathsf{op}}<1$, let $\gamma_K$ be the joint centered circular complex Gaussian law of $(Z,W)\in\bbC^2\times\bbC^2$ with covariance
\[
 \begin{pmatrix}I_2&K\\K^*&I_2\end{pmatrix}.
\]
Write $\mathcal R_K=d\gamma_K/d(\gamma_2\otimes\gamma_2)$, $\varrho=\|K\|_{\mathsf{op}}$, and $\Delta=\det(I_2-KK^*)$.
For $|t|<1$ and a phase $\phi$, define orthonormal vectors
\[
 c_{t,\phi}=\begin{pmatrix}\sqrt{(1+t)/2}\\ e^{i\phi}\sqrt{(1-t)/2}\end{pmatrix},\qquad
 b_{t,\phi}=\begin{pmatrix}\sqrt{(1-t)/2}\\-e^{i\phi}\sqrt{(1+t)/2}\end{pmatrix}.
\]
Let $\nu_t$ be the probability measure obtained by choosing $\phi$ uniformly and independently choosing a scalar complex variable $X$ with density
\[
 [(1-\eta)+\eta|x|^2]\pi^{-1}e^{-|x|^2},
\]
and setting $Z=c_{t,\phi}X$. Under $\nu_t$, $T(Z)=t$, the two phases are independent uniform variables, and $S(Z)$ has density $[(1-\eta)+\eta s]e^{-s}$ on $(0,\infty)$. Thus this is a precise probability measure on the cone, requiring no formal Dirac-delta manipulations. Set
\[
 q_\eta(K;t,s)=\iint\mathcal R_K(z,w)\,\nu_t(dz)\nu_s(dw).
\]

\begin{lem}\label{lem:cone}
For $|t|,|s|\le1/4$,
\begin{equation}\label{eq:cone-bound}
 q_\eta(K;t,s)\le
 \exp\left(t^2+s^2-\frac{\varrho^2}{2000}\right)\Delta^{-1/4}.
\end{equation}
\end{lem}

\begin{proof}
All phase averages below are over independent uniform phases $\phi,\psi$. Express $K$ in the orthonormal bases $(c_{t,\phi},b_{t,\phi})$ and $(c_{s,\psi},b_{s,\psi})$:
\[
 \begin{pmatrix}d&e\\ f&c\end{pmatrix}.
\]
To justify integration over the cone, let $p_{\phi,\psi}(x,\zeta,y,\omega)$ be the Lebesgue density of $\gamma_K$ in these orthonormal coordinates. Then
\[
 \mathcal R_K(c_{t,\phi}x,c_{s,\psi}y)
 \pi^{-2}e^{-|x|^2-|y|^2}
 =\pi^2 p_{\phi,\psi}(x,0,y,0).
\]
This is a pointwise identity between Gaussian densities. Integrating in $x,y$ expresses the cone integral in terms of the density of the two constrained coordinates at zero and the conditional moments of the remaining coordinates.

Put $u_0=1-|c|^2$. The joint density at zero of the two constrained Gaussian coordinates $b_{t,\phi}^*Z,b_{s,\psi}^*W$, multiplied by $\pi^2$, is $u_0^{-1}$. Conditional on those coordinates being zero, the remaining scalar complex Gaussian coordinates have covariance
\[
 \sigma_1=1-\frac{|e|^2}{u_0},\qquad
 \sigma_2=1-\frac{|f|^2}{u_0},\qquad
 \sigma_{12}=d+\frac{e\bar c f}{u_0}.
\]
This follows by taking the Schur complement of $\bigl(\begin{smallmatrix}1&c\\\bar c&1\end{smallmatrix}\bigr)$. In particular, $0\le\sigma_1,\sigma_2\le1$ and $|\sigma_{12}|^2\le\sigma_1\sigma_2\le1$.
Define
\[
 \mathcal I=\bbE_{\phi,\psi}\frac1{u_0},\qquad
 \mathcal D=\bbE_{\phi,\psi}\frac{|e|^2+|f|^2}{u_0^2},\qquad
 \mathcal Q=\bbE_{\phi,\psi}\frac{\sigma_1\sigma_2+|\sigma_{12}|^2}{u_0}.
\]
The identity $\bbE|X|^2|Y|^2=\sigma_1\sigma_2+|\sigma_{12}|^2$ for a centered circular complex Gaussian pair follows by expanding into independent standard complex Gaussian variables. Applying it to the two radial weights in the definition of $\nu_t,\nu_s$ gives
\begin{equation}\label{eq:cone-decomposition}
 q_\eta(K;t,s)=\mathcal I-\eta(1-\eta)\mathcal D+\eta^2(\mathcal Q-\mathcal I).
\end{equation}
We next prove three estimates for these phase averages.

\smallskip
\emph{The basic phase integral.}
For fixed $v\in\bbC^2$ with $x=\|v\|^2<1$, define
\[
 J_t(v)=\bbE_\phi\frac1{1-|b_{t,\phi}^*v|^2}.
\]
Let $u=1-x/2+t(|v_1|^2-|v_2|^2)/2$. The elementary integral
$\frac1{2\pi}\int_0^{2\pi}(u-a\cos\phi)^{-1}d\phi=(u^2-a^2)^{-1/2}$ gives $J_t(v)=R^{-1}$, where
\begin{equation}\label{eq:phase-square}
 R^2=(1-t^2)(1-x)+
 \left(\frac{|v_1|^2-|v_2|^2}{2}+t(1-x/2)\right)^2.
\end{equation}
In particular, $J_t(v)\le(1-t^2)^{-1/2}(1-x)^{-1/2}$. Using this with $v=Kb_{s,\psi}$ and then Cauchy--Schwarz yields
\[
 \mathcal I\le(1-t^2)^{-1/2}
 \left(\bbE_\psi\frac1{b_{s,\psi}^*(I_2-K^*K)b_{s,\psi}}\right)^{1/2}.
\]
For any positive Hermitian $2\times2$ matrix $B$, the same circle integral gives
\[
 \bbE_\psi\frac1{b_{s,\psi}^*Bb_{s,\psi}}
 \le\bigl[(1-s^2)\det B\bigr]^{-1/2}.
\]
More explicitly, if $B=(B_{jk})_{j,k=1}^2$, then this phase average equals
\[
 \left[(1-s^2)\det B+
 \frac{\bigl((1-s)B_{11}-(1+s)B_{22}\bigr)^2}{4}\right]^{-1/2}.
\]
Therefore
\begin{equation}\label{eq:I-bound}
 \mathcal I\le(1-t^2)^{-1/2}(1-s^2)^{-1/4}\Delta^{-1/4}.
\end{equation}

\smallskip
\emph{The quadratic radial term.}
If $\varrho\le1/2$, then
$|\sigma_{12}|\le\varrho+\varrho^3/(1-\varrho^2)\le4\varrho/3$.
If $\varrho\ge1/2$, the bound $|\sigma_{12}|^2\le1\le4\varrho^2$ applies. Since $\sigma_1\sigma_2\le1$, both cases imply
\begin{equation}\label{eq:Q-bound}
 \mathcal Q-\mathcal I\le4\varrho^2\mathcal I.
\end{equation}

\smallskip
\emph{A lower bound for the radial correction.}
Fix $\psi$, put $v=Kb_{s,\psi}$ and $x=\|v\|^2$, and keep the notation $u,R$ from \eqref{eq:phase-square}. Since $|e|^2+|c|^2=x$ and $\bbE_\phi u_0^{-2}=u/R^3$, we have
\[
 \bbE_\phi\frac{|e|^2}{u_0^2}
 =J_t(v)\left(1-\frac{(1-x)u}{R^2}\right).
\]
The bounds $R^2\ge(1-t^2)(1-x)$ and $u\le1-(1-|t|)x/2$ show, for $|t|\le1/4$, that
\[
 1-\frac{(1-x)u}{R^2}
 \ge \frac{x}{2(1+|t|)}-\frac{t^2}{1-t^2}
 \ge \frac25x-2t^2.
\]
Let $\mathcal T=\bbE_\psi[xJ_t(Kb_{s,\psi})]$. We claim that $\mathcal T\ge\varrho^2\mathcal I/8$. First, $J_t\ge1$ and
\[
 \bbE_\psi x\ge\frac{1-|s|}{2}\|K\|_{\mathsf{F}}^2\ge\frac38\varrho^2,
\]
so $\mathcal T\ge3\varrho^2/8$. Second, on $x\le\varrho^2/4$ we have $J_t\le8/\sqrt{45}<4/3$, and hence
\[
 \mathcal T\ge\frac{\varrho^2}{4}\left(\mathcal I-\frac43\right).
\]
The first bound proves the claim when $\mathcal I\le3$; the second proves it when $\mathcal I\ge3$. It follows that
\[
 \bbE_{\phi,\psi}\frac{|e|^2}{u_0^2}
 \ge\left(\frac{\varrho^2}{20}-2t^2\right)\mathcal I.
\]
The same calculation with the roles reversed bounds the contribution of $|f|^2$. Thus
\begin{equation}\label{eq:D-bound}
 \mathcal D\ge\left(\frac{\varrho^2}{10}-2(t^2+s^2)\right)\mathcal I.
\end{equation}

Finally, plugging \eqref{eq:Q-bound} and \eqref{eq:D-bound} into \eqref{eq:cone-decomposition}, since
\[
 \frac{\eta(1-\eta)}{10}-4\eta^2
 =0.00059>\frac1{2000},
\]
we obtain
\[
 q_\eta(K;t,s)\le\mathcal I
 \exp\left[-\frac{\varrho^2}{2000}+2\eta(1-\eta)(t^2+s^2)\right].
\]
For $0\le y\le1/16$, $-\log(1-y)\le16y/15$. Hence the logarithm of the extra factors in \eqref{eq:I-bound}, together with the positive term in this exponent, is at most
\[
 \left(\frac8{15}+2\eta(1-\eta)\right)t^2+
 \left(\frac4{15}+2\eta(1-\eta)\right)s^2
 \le t^2+s^2.
\] This proves \eqref{eq:cone-bound}.
\end{proof}

\subsection{Uniform kernels and the $L^2$ comparison}

In this subsection we prove Proposition~\ref{prop:L2}.
For $a,b\in\{g,r\}$, put
\[
 \rho_{ab}(K)=\bbE_{\gamma_K}[a(Z)b(W)].
\]
The following facts isolate exactly what is needed from the cone estimate.

\begin{lem}\label{lem:kernels}
There are absolute constants $c,C>0$ such that, for all sufficiently large $M$ and all $\|K\|_{\mathsf{op}}<1$,
\begin{equation}\label{eq:global-kernel}
 0\le\rho_{ab}(K)\le
 \exp\bigl(\zeta_M-c\|K\|_{\mathsf{op}}^2\bigr)\det(I_2-KK^*)^{-1/4},
 \qquad \zeta_M=2\delta+2\varepsilon^2,
\end{equation}
and $\rho_{ab}(K)\le C M^{52}$. Moreover, uniformly on a sufficiently small fixed neighborhood of zero,
\begin{equation}\label{eq:local-kernel}
 \rho_{ab}(K)=1+(1-v_M)^2\|K\|_{\mathsf{F}}^2+O(\|K\|_{\mathsf{F}}^4).
\end{equation}
The constants in the remainder are independent of $M$.
\end{lem}

\begin{proof}
The law $g\,d\gamma_2$ is the average of the cone laws $\nu_t$ over $t\in[-\varepsilon,\varepsilon]$, restricted to $S\ge\delta$ and divided by $c_\delta$. Since $\mathcal R_K\ge0$, deleting those two radial restrictions can only increase the integral defining $\rho_{gg}$. Lemma~\ref{lem:cone} gives
\[
 \rho_{gg}(K)\le c_\delta^{-2}
 \exp\left(2\varepsilon^2-\frac{\|K\|_{\mathsf{op}}^2}{2000}\right)\Delta^{-1/4}
 \le\exp\left(2\delta+2\varepsilon^2-\frac{\|K\|_{\mathsf{op}}^2}{2000}\right)\Delta^{-1/4}.
\]
For the reference density, let $b_M=1-v_M\to(1-\eta)/2<1/2$. Gaussian integration gives exactly
\begin{equation}\label{eq:reference-kernel}
 \rho_{rr}(K)=\det(I_2-b_M^2KK^*)^{-1}.
\end{equation}
For $0\le x<1$ and $b_M^2<1/4$, comparison of the logarithmic power series yields
\[
 -\log(1-b_M^2x)+\frac14\log(1-x)
 \le-\left(\frac14-b_M^2\right)x.
\]
Applying this to the eigenvalues of $KK^*$ proves \eqref{eq:global-kernel} for $rr$ with a fixed positive constant $c$.

For the mixed kernel, take a singular-value decomposition $K=U_0\Lambda V_0^*$, where $\Lambda$ is diagonal with entries in $[0,1)$, and represent
\begin{align*}
 Z&=U_0\bigl(\Lambda^{1/2}X+(I_2-\Lambda)^{1/2}X_1\bigr),\\
 W&=V_0\bigl(\Lambda^{1/2}X+(I_2-\Lambda)^{1/2}X_2\bigr),
\end{align*}
using three independent standard complex Gaussian two-vectors. Put $f(X)=\bbE[g(Z)\mid X]$ and $h(X)=\bbE[r(W)\mid X]$. Conditional independence and Cauchy--Schwarz give
\[
    \rho_{gr}(K)=\bbE[f(X)h(X)]\le(\bbE f(X)^2\,\bbE h(X)^2)^{1/2}
\]
To identify the two second moments, replace $X_1$ or $X_2$ by an independent copy while keeping $X$ fixed. The resulting pairs have cross-covariances $U_0\Lambda U_0^*=(KK^*)^{1/2}$ and $V_0\Lambda V_0^*=(K^*K)^{1/2}$, respectively. Therefore
\[
 \rho_{gr}(K)\le
 \left[\rho_{gg}\bigl((KK^*)^{1/2}\bigr)
       \rho_{rr}\bigl((K^*K)^{1/2}\bigr)\right]^{1/2}.
\]
The two matrices on the right have the same singular values as $K$. This proves the mixed bound; the reverse mixed kernel follows from $\rho_{rg}(K)=\rho_{gr}(K^*)$. The uniform polynomial bound follows from $\max(\|g\|_\infty,\|r\|_\infty)\le CM^{52}$ and the unit mass of the other factor.

It remains to verify the local expansion. Under the product tilted law $a\,d\gamma_2\otimes b\,d\gamma_2$, the two vectors are independent, centered, circular, and have covariance $v_M I_2$. Expanding the Gaussian density ratio gives, through degree two in $K$,
\begin{align*}
 \mathcal R_K(z,w)={}&1+2\operatorname{Re}(z^*Kw)
 +\mathsf{tr}(KK^*)-z^*KK^*z-w^*K^*Kw\\
 &\hspace{18mm}+2\bigl(\operatorname{Re}(z^*Kw)\bigr)^2
 +\text{higher-order terms}.
\end{align*}
The linear expectation vanishes, and the quadratic expectation is
\[
 (1-2v_M+v_M^2)\|K\|_{\mathsf{F}}^2.
\]
All odd homogeneous terms vanish because each tilted measure is invariant under multiplication by $-1$; equivalently, $\rho_{ab}(-K)=\rho_{ab}(K)$. To justify a uniform fourth-order remainder, use the exact ratio
\begin{align*}
 \mathcal R_K(z,w)=\Delta^{-1}\exp\bigl(&-z^*[(I_2-KK^*)^{-1}-I_2]z
 -w^*[(I_2-K^*K)^{-1}-I_2]w\\
 &+2\operatorname{Re}[z^*K(I_2-K^*K)^{-1}w]\bigr).
\end{align*}
On a sufficiently small fixed matrix ball, its fourth directional derivatives are bounded by a fixed polynomial in $S(z)+S(w)$ times $\exp((S(z)+S(w))/8)$. The uniformly bounded exponential moments established after \eqref{eq:variance} justify differentiation under the integral and the claimed uniform remainder.
\end{proof}

\begin{lem}\label{lem:overlap}
For $M\ge4$ and independent Haar orthonormal two-frames $U,V\in\bbC^{M\times2}$, the overlap $K=U^*V$ has density
\begin{equation}\label{eq:overlap-density}
 d\mu_M(K)=c_M\det(I_2-KK^*)^{M-4}\,dK,
 \qquad c_M=\frac{(M-1)(M-2)^2(M-3)}{\pi^4},
\end{equation}
on $\|K\|_{\mathsf{op}}<1$, where $dK$ is Lebesgue measure on $\bbC^{2\times2}\simeq\bbR^8$.
\end{lem}

\begin{proof}
Fix $U$ to be the first two coordinate vectors by unitary invariance. Let $z,w$ be the first two coordinates of the first and second columns of $V$. The density of $z$ is
\[
 \frac{(M-1)(M-2)}{\pi^2}(1-\|z\|^2)^{M-3}\mathbbm{1}_{\|z\|<1}.
\]
This is the usual spherical projection density, obtained by writing a uniform complex unit vector as a normalized Gaussian vector and integrating its remaining $M-2$ coordinates. Conditional on the first column, the second is uniform on the unit sphere in its $(M-1)$-dimensional orthogonal complement. Its first-two-coordinate projection is an elliptic spherical projection with shape matrix $B=I_2-zz^*$, and therefore has density
\[
 \frac{(M-2)(M-3)}{\pi^2\det B}
 (1-w^*B^{-1}w)^{M-4}\mathbbm{1}_{w^*B^{-1}w<1}.
\]
The complex linear change of variables has real Jacobian $\det B$. Multiplying the two densities and using
$\det B=1-\|z\|^2$ and
$\det(B-ww^*)=\det B(1-w^*B^{-1}w)$ gives \eqref{eq:overlap-density}.
\end{proof}

We now finish the proof of Proposition~\ref{prop:L2}.
\begin{proof}
Independence of the rows and Lemma~\ref{lem:overlap} give
\begin{equation}\label{eq:second-moments}
 \bbE_0[L_aL_b]=\int\rho_{ab}(K)^N\,d\mu_M(K).
\end{equation}
Write $x=\|K\|_{\mathsf{F}}^2$ and $b_M=1-v_M$. Since $4b_M^2\to(1-\eta)^2<1$, the local expansion in Lemma~\ref{lem:kernels} permits a fixed $r_0>0$ such that, on $\|K\|_{\mathsf{op}}\le r_0$, the product of the overlap density and any kernel raised to power $N$ is bounded by $CM^4e^{-cMx}$. Indeed, use $\log\Delta\le-x$ and
\[
 \log\rho_{ab}(K)\le b_M^2x+Cx^2
\]
and decrease $r_0$ until the quartic term is absorbed by the fixed gap $1-4b_M^2$. The local expansion also allows us to ensure $1/2\le\rho_{ab}(K)\le2$ on this ball, uniformly in $M$. Thus the same density bound holds with the exponent $N-1$ in place of $N$, which will be used below.

The three kernels have the same quadratic term, so
\[
 |\rho_{ab}(K)^N-\rho_{rr}(K)^N|
 \le CNx^2\max(\rho_{ab}(K),\rho_{rr}(K))^{N-1}.
\]
Consequently, their integrated local differences are at most
\[
 CNM^4\int_{\bbR^8}\|K\|_{\mathsf{F}}^4e^{-cM\|K\|_{\mathsf{F}}^2}\,dK
 =O(M^{-1}).
\]
Here the integral is $O(M^{-6})$ by rescaling the eight real coordinates.

For completeness, the boundary of the matrix ball must also be controlled. Put $B_M=CM^{52}$, enlarged to dominate all three kernels. For $\|K\|_{\mathsf{op}}>r_0$, apply the polynomial bound to twelve factors and \eqref{eq:global-kernel} to the other $N-12$ factors. Since $N=4M-5$,
\begin{align*}
 \rho_{ab}(K)^N\Delta^{M-4}
 &\le B_M^{12}
 e^{(N-12)\zeta_M-c(N-12)r_0^2}
 \Delta^{M-4-(N-12)/4}\\
 &=B_M^{12}
 e^{(N-12)\zeta_M-c(N-12)r_0^2}\Delta^{1/4}.
\end{align*}
The matrix ball has fixed finite volume, $c_M=O(M^4)$, and $N\zeta_M=O(M^{-1})$. Thus each tail integral is at most a polynomial in $M$ times $e^{-c'M}$ and is $o(M^{-1})$.

Combining these estimates in \eqref{eq:second-moments} gives
\[
 \bbE_0(L_g-L_r)^2
 =\int[\rho_{gg}(K)^N+\rho_{rr}(K)^N-2\rho_{gr}(K)^N]d\mu_M(K)
 =O(M^{-1}).
\]
No pointwise positivity of the displayed integrand is needed: bound its absolute value by
$|\rho_{gg}^N-\rho_{rr}^N|+2|\rho_{gr}^N-\rho_{rr}^N|$.
\end{proof}

\subsection{An exact ambiguity under the planted law}

We next prove Proposition~\ref{prop:plant}. By right-unitary invariance it suffices to condition on the planted frame $U=(e_1,e_2)$. Write the rows as $(a_i^\varepsilon)^*$. Their adjoint column vectors can be generated as
\begin{equation}\label{eq:planted-row}
 a_i^\varepsilon=
 \left(\sqrt{\frac{S_i(1+\xi_i)}2}e^{i\alpha_i},\,
       \sqrt{\frac{S_i(1-\xi_i)}2}e^{i\beta_i},\,w_i^{\mathsf T}\right)^{\mathsf T},
\end{equation}
where $S_i$ has density \eqref{eq:radial}, $\xi_i$ is uniform on $[-\varepsilon,\varepsilon]$, $\alpha_i,\beta_i$ are uniform phases, and $w_i\sim\gamma_{M-2}$, all independently. Let $a_i^0$ be the same vector with $\xi_i$ replaced by zero. For $Q_0$ from \eqref{eq:def-Q-0}, we have $(a_i^0)^*Q_0a_i^0=0$.

We use the same parameters and the same rank-two parametrization as in \eqref{eq:def-matrix-D-C}--\eqref{eq:def-Q(s,b,z,t)}:
\begin{equation}\label{eq:chart}
\begin{gathered}
 D(s,b)=\begin{pmatrix}1+s/2&b\\\overline b&-1+s/2\end{pmatrix},
 \qquad C(z,t)=\begin{pmatrix}z&t\end{pmatrix},\\
 Q_\theta=Q(s,b,z,t)=
 \begin{pmatrix}D(s,b)&C(z,t)^*\\
 C(z,t)&C(z,t)D(s,b)^{-1}C(z,t)^*\end{pmatrix}.
\end{gathered}
\end{equation}
Here $\theta$ consists of $s$, the real and imaginary parts of $b$, and the real and imaginary parts of $z,t$, with
\[
 \|\theta\|^2=s^2+|b|^2+\|z\|^2+\|t\|^2.
\]
Its real dimension is $3+4(M-2)=4M-5=N$.
On a sufficiently small fixed ball about zero, $D=D(s,b)$ is invertible with one positive and one negative eigenvalue. Writing $C=C(z,t)$, the factorization
\[
 Q_\theta=\begin{pmatrix}I_2\\CD^{-1}\end{pmatrix}
 D\begin{pmatrix}I_2&D^{-1}C^*\end{pmatrix}
\]
shows that $Q_\theta$ has rank two and the same inertia as $D$.
Define $F^\varepsilon:\bbR^N\to\bbR^N$ on this ball by
\[
 F_i^\varepsilon(\theta)=\frac{(a_i^\varepsilon)^*Q_\theta a_i^\varepsilon}{S_i}.
\]
Then $F^0(0)=0$ and $F_i^\varepsilon(0)=\xi_i$. Let $J_0=DF^0(0)$.

\begin{lem}\label{lem:smallball}
The independent real rows of $J_0$ satisfy
\begin{equation}\label{eq:smallball}
 \sup_{\|v\|=1}\bbP\bigl(|(J_0)_{i,\cdot}v|\le u\bigr)\le Cu^{1/3},
 \qquad 0<u<1,
\end{equation}
uniformly in $M$.
\end{lem}

\begin{proof}
Write a unit coefficient vector as $(\alpha,\beta,\gamma,p,q)$, where $\alpha,\beta,\gamma\in\bbR$ and $p,q\in\bbC^{M-2}$. Differentiating \eqref{eq:chart} at zero and absorbing the common phase into $w_i$, its scalar product with a row of $J_0$ has the law
\begin{equation}\label{eq:row-derivative}
 \frac\alpha2+\beta\cos\phi-\gamma\sin\phi
 +\sqrt{\frac2S}\operatorname{Re}\bigl[w^*(p+e^{i\phi}q)\bigr],
\end{equation}
where $\phi$ is uniform, $w\sim\gamma_{M-2}$, and $S$ has \eqref{eq:radial}, independently. Also $\bbE S^{1/4}\le C$.

Suppose first that $\|p\|^2+\|q\|^2\ge1/4$. Conditional on $S,\phi$, the last term in \eqref{eq:row-derivative} is a real centered Gaussian of variance $V(\phi)/S$, where
\[
 V(\phi)=\|p+e^{i\phi}q\|^2
 =\|p\|^2+\|q\|^2+2\operatorname{Re}(e^{i\phi}p^*q).
\]
The elementary cosine sublevel estimate gives $\bbP_\phi(V(\phi)\le t^2)\le Ct$. To see uniformity, the constant term is at least $1/4$ and the oscillation amplitude is at most that constant; the worst sublevel set occurs at a quadratic zero of the cosine. If the oscillation amplitude is less than half the constant term, the sublevel set is empty for all sufficiently small $t$.
Splitting according to $V\le t^2$ and using the Gaussian interval bound on the complement gives
\[
 \bbP\bigl(|(J_0)_{i,\cdot}v|\le u\mid S\bigr)
 \le Ct+C\frac{u\sqrt S}{t}.
\]
Taking $t=(u\sqrt S)^{1/2}$ and averaging over $S$ yields $Cu^{1/2}$, which is sufficient.

Otherwise, $\alpha^2+\beta^2+\gamma^2>3/4$. The trigonometric polynomial
$m(\phi)=\alpha/2+\beta\cos\phi-\gamma\sin\phi$ satisfies
\[
 \bbP_\phi(|m(\phi)|\le t)\le C\sqrt t.
\]
Indeed, if $\sqrt{\beta^2+\gamma^2}$ is bounded below, this is the arcsine-density interval estimate. If that amplitude is sufficiently small, $|\alpha|/2$ is bounded away from zero and the assertion is immediate.
For $|m|\ge t\ge2u$ and any $\sigma\ge0$, a one-dimensional Gaussian satisfies
\[
 \bbP(|m+\sigma G|\le u)\le C\frac{u}{t}.
\]
For $\sigma>0$, maximize the Gaussian density over $\sigma$ at points of distance at least $|m|-u\ge|m|/2$ from zero; the case $\sigma=0$ is immediate. Taking $t=u^{2/3}$ proves the bound $C\sqrt t+Cu/t\le Cu^{1/3}$ for small $u$, and increasing $C$ handles the remaining $u$.
\end{proof}

\begin{lem}\label{lem:singular}
For $0<\kappa<1/\sqrt N$,
\[
 \bbP(s_{\min}(J_0)\le\kappa)\le CN(\sqrt N\,\kappa)^{1/3}.
\]
In particular, with $\kappa=M^{-12}$ the failure probability is $O(M^{-17/6})$.
\end{lem}

\begin{proof}
For each row, condition on all the other rows and choose measurably a unit vector orthogonal to their span (for example, by projecting the first standard basis vector with nonzero orthogonal projection and normalizing). This normal is independent of the remaining row. Lemma~\ref{lem:smallball} bounds the probability that its scalar product with that row has magnitude at most $u$ by $Cu^{1/3}$. Therefore the same upper bound applies to the event that the distance $d_i$ of that row from the span of the others is at most $u$.
For an invertible square matrix, the norm of the $i$th column of its inverse is $d_i^{-1}$. Hence
\[
 s_{\min}(J_0)\ge\frac{\min_i d_i}{\sqrt N}.
\]
For a singular square matrix at least one $d_i$ is zero, so the same implication used in the probability bound remains valid. A union bound with $u=\sqrt N\,\kappa$ proves the lemma.
\end{proof}

\begin{proof}[Proof of Proposition~\ref{prop:plant}]
The radial density \eqref{eq:radial} has a uniform exponential moment, and $w_i$ is Gaussian. Thus, with probability at least $1-CMe^{-cM}$,
\begin{equation}\label{eq:row-norm}
 \max_i\|a_i^\varepsilon\|^2=\max_i\|a_i^0\|^2\le CM.
\end{equation}
Also $S_i\ge\delta=M^{-2}$ deterministically. On a fixed small ball in the chart \eqref{eq:chart}, the first two derivatives of $Q_\theta$, as maps from its Euclidean coordinate space to operator-norm matrices, are bounded by absolute constants. This follows directly by differentiating $CD^{-1}C^*$ and using the uniform bound for $D^{-1}$.
Consequently, on \eqref{eq:row-norm}, the Jacobian of $F^\varepsilon$ is Lipschitz with constant
\begin{equation}\label{eq:Lipschitz}
 L\le C\sqrt N\,M/\delta\le CM^{7/2}.
\end{equation}
Furthermore, $\|a_i^\varepsilon-a_i^0\|\le C\varepsilon\sqrt{S_i}$, so
\begin{equation}\label{eq:Jac-perturb}
 \|DF^\varepsilon(0)-J_0\|_{\mathsf{op}}\le C\varepsilon M^2,
 \qquad \|F^\varepsilon(0)\|\le\varepsilon\sqrt N.
\end{equation}
For the first bound, each row-gradient difference has norm at most
$C(\|a_i^\varepsilon\|+\|a_i^0\|)\|a_i^\varepsilon-a_i^0\|/S_i
\le C\varepsilon\sqrt{M/\delta}$; then take the Frobenius norm over $N$ rows.

Intersect \eqref{eq:row-norm} with $s_{\min}(J_0)\ge\kappa=M^{-12}$. Lemma~\ref{lem:singular} shows that this event has probability at least $1-O(M^{-17/6})-CMe^{-cM}$. On it, \eqref{eq:Jac-perturb} implies that $J_\varepsilon=DF^\varepsilon(0)$ is invertible and $\|J_\varepsilon^{-1}\|_{\mathsf{op}}\le2\kappa^{-1}$.

Consider the map $\Psi(\theta)=\theta-J_\varepsilon^{-1}F^\varepsilon(\theta)$ on the closed ball of radius
\[
 R_M=4\kappa^{-1}\varepsilon\sqrt N.
\]
Its value at zero has norm at most $R_M/2$, and \eqref{eq:Lipschitz} gives
\[
 \sup_{\|\theta\|\le R_M}\|D\Psi(\theta)\|_{\mathsf{op}}
 \le2\kappa^{-1}LR_M
 \le C\kappa^{-2}\varepsilon M^4=O(M^{-22})<\frac12.
\]
Also $R_M\to0$, so this ball lies in the fixed chart. Thus $\Psi$ maps the ball to itself and is a contraction. Its fixed point $\theta_*$ satisfies $F^\varepsilon(\theta_*)=0$.
The matrix $Q_{\theta_*}$ is an indefinite rank-two Hermitian matrix annihilated by all the measurement forms. It produces an exact ambiguity by Lemma~\ref{lem:equivalent-formulation}. Therefore
\[
 \bbP_g(\mathcal E_M)\le O(M^{-17/6})+CMe^{-cM}=O(M^{-2}). \qedhere
\]
\end{proof}

\bibliographystyle{alpha}
\bibliography{phaseretrieval}

\end{document}